\documentclass[11pt]{article}
\usepackage{amsfonts,epsf,amsmath,amssymb}
\usepackage{tikz}
\usepackage[ruled,vlined,linesnumbered]{algorithm2e}
\usepackage{color,enumerate}

\newtheorem{theorem}{\bf Theorem}[section]
\newtheorem{corollary}[theorem]{\bf Corollary}

\newtheorem{proposition}[theorem]{\bf Proposition}

\newcommand{\qed}{\hfill $\square$ \bigskip}

\newcommand{\ggr}{\gamma_{gr}}

\begin{document}

\title{Dominating sequences under atomic changes with applications in Sierpi\'{n}ski and interval graphs}

\author{
Bo\v{s}tjan Bre\v{s}ar$^{a,b}$
\and
Tanja Gologranc$^{a}$
\and
Tim Kos$^{b}$
}

\date{}

\maketitle

\begin{center}
$^a$ Faculty of Natural Sciences and Mathematics, University of Maribor, Slovenia

$^b$ Institute of Mathematics, Physics and Mechanics, Ljubljana, Slovenia

\end{center}

\begin{abstract}
A sequence $S=(v_1,\ldots,v_k)$ of distinct vertices of a graph $G$ is called a legal sequence if $N[v_i] \setminus \cup_{j=1}^{i-1}N[v_j]\not=\emptyset$ for any $i$. 
The maximum length of a legal (dominating) sequence in $G$ is called the Grundy domination number
$\gamma_{gr}(G)$ of a graph $G$. It is known that the problem of determining the Grundy domination number is NP-complete in general, while efficient algorithm exist for trees and some other classes of graphs \cite{bgmrr-2014}. In this paper we find an efficient algorithm for the Grundy domination number of an interval graph. We also show the exact value of the Grundy domination number of an arbitrary Sierpi\'{n}ski graph $S_p^n$, and present algorithms to construct the corresponding sequence. These results are obtained by using the       main result of the paper, which are sharp bounds for the Grundy domination number of a vertex- and edge-removed graph. That is, given a graph $G$,  $e\in E(G)$, and $u\in V(G)$, we prove that $\ggr(G)-1\le \ggr(G-e) \le \ggr(G)+1$ and $\ggr(G)-2\le \ggr(G-u) \le \ggr(G)$. For each of the bounds there exist graphs, in which all three possibilities occur for different edges, respectively vertices. 
\end{abstract}

\noindent {\bf Key words:} Grundy domination number; Sierpi\'nski graphs; interval graphs.

\medskip\noindent
{\bf AMS Subj.\ Class:} 05C69

\section{Introduction}

Domination is one of the oldest and most studied topics in graph theory,
and is known for many, also real-world applications. Domination theory
was comprehensively surveyed in two monographs from almost 20 years ago~\cite{hhs-1998,hhs2-1998}, while a recent monograph~\cite{hy-2013} focuses
only on the total domination, one of the most basic concepts in the theory.
Many other variations of the classical domination number of a graph have been
introduced in years. A very recent and natural one is the so-called 
Grundy domination number, which somehow describes the worst 
scenario that can happen when a dominating set is built. More precisely, 
the Grundy domination number of a graph is the maximum length of a sequence of its vertices,
such that each time a vertex is added it dominates some
vertex that was not dominated by previous vertices in the sequence~\cite{bgmrr-2014}.
(A total version of Grundy domination was introduced in~\cite{bhr-2016}.)

One of the focuses in this paper will be on two rather famous classes of graphs,
namely the interval graphs and the Sierpi\'{n}ski graphs. 
The interval graphs, i.e., the intersection graphs of intervals in the real line,
have been introduced by Hajos about 60 years ago; they have many applications
in large diversity of areas such as archeology, artificial intelligence, economics and planning, cf.~\cite{bls-99,mcmc-99}, and probably most intensively in mathematical biology, 
see e.g.~\cite{acfgm-07,cr-2006,wg-1986}. On the other hand, Sierpi\'{n}ski graphs were 
also introduced in relation with various problems, such as Tower of Hanoi game, physics, interconnection networks, and topology; see more details in a very recent extensive survey~\cite{kmz-16+}. The main feature of Sierpi\'{n}ski graphs is their fractal-like nature, and can be considered as a basic discrete version of fractals. 

While domination and total domination number are computationally hard problems in general graphs~\cite{GJ}, they can be efficiently determined in several classes of graphs, in particular, in the interval graphs and the Sierpi\'nski graphs. An efficient algorithm for computing the domination number (and some related invariants) in strongly chordal graphs, which as a subclass contain the interval graphs, was first designed by Farber in the 1980s~\cite{fa-1984}; Keil soon followed with a linear time algorithm for determining the total domination number of interval graphs~\cite{ke-1986}. Even more can be said in the case of Sierpi\'nski graphs, since exact formulas for the domination numbers of these graphs were established in~\cite{kmp-2003}, and more recently also the exact total domination numbers of arbitrary Sierpi\'nski graphs were proven~\cite{gkmmp-13}.

A motivation for studying dominating sequences was a domination game as introduced in~\cite{brklra-2010}. Two players alternate turns in this game, one player wants to build a dominating set as quickly as possible, while the other (called Staller) wants to delay the process. By definition of legality of moves in the game it follows that the resulting number of moves when both players play optimally, called the game domination number of a graph, is bounded above by the Grundy domination number (in fact, a legal sequence whose length is the Grundy domination number is obtained when only Staller plays the game). The domination game has been intensively studied by several authors, and a lot of efforts were given to resolve the (still open) $3/5$-conjectures from~\cite{bill-2012}. In~\cite{brdokl-2014} the authors examined the possible changes of the game domination number under vertex- and edge-removal in a graph, and proposed a classification of the graphs with respect to the corresponding behaviour. (For a very recent paper on game domination see~\cite{kks-16+}.)

In this paper we describe the behaviour of the Grundy domination number when an edge or a vertex is removed from a graph, see Section~\ref{s:deliting}. We prove that in any graph $G$ and $u\in V(G)$, the Grundy domination number of $G$ drops by at most 2 when $u$ is removed from $G$. Next, if $e$ is an edge of an arbitrary graph $G$, then the Grundy domination number of $G-e$ is between one less than the Grundy domination number of $G$ and one more than that number. Combining the edge-removal bound and the recursive fractal structure of Sierpi\'nski graphs, we prove in Section~\ref{s:Sierpinski} that the Grundy domination number of the Sierpi\'nski graph $S_p^n$ equals $p^{n-1} + \frac{p(p^{n-1}-1)}{2}$. In addition, we present two efficient algorithms to construct a Grundy dominating sequence of a Sierpi\'nski graph. The first algorithm is optimal, because it uses a recursive formula that builds only the Sierpi\'nski labels of all vertices of the Grundy dominating sequence; the second algorithm is nice in the sense that the vertices are ordered lexicographically with respect to their Sierpi\'nski labels (hence, one could only follow this order and decide whether a given vertex can be put in the sequence or not). Finally, in Section~\ref{s:interval} we make use of the vertex-removal formula (in fact, a version of this formula for the removal of a simplicial vertex) to construct an efficient algorithm for determining a Grundy domination number (resp. sequence) of an arbitrary interval graph. 
In the remainder of this section we present main formal definitions and notation, used throughout the paper.

%%%%%%%%%%%%%%%%

Let $S=(v_1,\ldots,v_k)$ be a sequence of distinct vertices of a graph $G$. The corresponding set $\{v_1,\ldots,v_k\}$
of vertices from the sequence $S$ will be denoted by $\widehat{S}$. A sequence $S=(v_1,\ldots,v_k)$, where $v_i\in V(G)$,
is called a {\em legal (closed neigborhood) sequence} if, for each $i$
$$N[v_i] \setminus \cup_{j=1}^{i-1}N[v_j]\not=\emptyset.$$
(We also say that $v_i$ is a {\em legal choice}, when the above inequality holds.)
If for a legal sequence $S$, the set $\widehat{S}$ is a dominating set of $G$, then $S$ is called a {\em dominating sequence} of $G$.
Adopting the notation from domination theory, each vertex $u\in N[v_i] \setminus \cup_{j=1}^{i-1}N[v_j]$
is called a {\em private neighbor} of $v_i$ with respect to $\{v_1,\ldots,v_{i}\}$. We will also use a more suggestive
term by saying that $v_i$ {\em footprints} the vertices from $N[v_i] \setminus \cup_{j=1}^{i-1}N[v_j]$, and that $v_i$ is
the {\em footprinter} of any $u\in N[v_i] \setminus \cup_{j=1}^{i-1}N[v_j]$.
For a dominating sequence $S$ any vertex in $V(G)$ has a unique footprinter in $\widehat{S}$.
Thus the function $f_S:V(G)\to \widehat{S}$ that maps each vertex to its footprinter is well defined.
Clearly the length $k$ of a dominating sequence $S=(v_1,\ldots,v_k)$ is bounded from below by the domination number $\gamma(G)$
of a graph $G$. We call the maximum length of a legal dominating sequence in $G$ the {\em Grundy domination number}
of a graph $G$ and denote it by $\gamma_{gr}(G)$. The corresponding sequence is called a {\em Grundy dominating sequence} of $G$ or $\gamma_{gr}$-sequence of $G$.
These concepts were introduced in \cite{bgmrr-2014}.

Let $S_1=(v_1,\ldots , v_n)$ and $S_2=(u_1,\ldots , u_m)$, $n,m \geq 0,$ be two sequences. The {\em concatenation} of $S_1$ and $S_2$ is defined as the sequence $S_1 \oplus S_2=(v_1,\ldots , v_n,u_1,\ldots , u_m).$ 

\section{Grundy domination number of subgraphs, obtained by edge- or vertex-deletion}\label{s:deliting}

\subsection{Edge-deletion}
First consider the subgraphs obtained by the smallest possible atomic change, i.e. deletion of  an edge. Unlike in the standard domination, where edge-deletion can only increase the domination number~\cite[Chapter 5]{hhs-1998}, the following possibilities appear for the Grundy domination number. 

\begin{theorem}
\label{thm:edge-removed}
If $G$ is a graph and $e\in E(G)$, then $$\ggr(G)-1\le \ggr(G-e) \le \ggr(G)+1.$$
Moreover, there exist graphs $G$ such that all values of $\ggr(G-e)$ between 
$\ggr(G)-1$ and $\ggr(G)+1$ are realized for different edges $e\in E(G)$.
\end{theorem}
\begin{proof}
Let $S$ be a Grundy dominating sequence of $G$, and let $e=uv$ be an edge, deleted from $G$.
Then, if $u,v\notin \widehat{S}$, $S$ is also a legal sequence of $G-e$.
In fact, the only case when $S$ is not a legal sequence of $G-e$ is when $f_S(u)=v$ and
$f_S^{-1}(v)=\{u\}$, or $f_S(v)=u$ and $f_S^{-1}(u)=\{v\}$. Without loss of generality, we may assume that $f_S(u)=v$ and $f_S^{-1}(v)=\{u\}$. Let $S'$ be the sequence obtained by removing the vertex $v$ from $S$. It is clear that $S'$ is a legal sequence in $G-e$ of length $|\widehat{S}|-1$. (If $S'$ is not a dominating sequence, we can append $u$ at the end of it, and obtain a legal dominating sequence of $G-e$.) We infer that $\ggr(G)-1\le \ggr(G-e)$.

For the other inequality consider a Grundy dominating sequence $S=(x_1,\ldots,x_k)$ of $G-e$, where $e=uv$ is deleted from $G$. If  $x_i$ is not in $N_{G-e}[u]\cup N_{G-e}[v]$, then it is clear that $N_G[x_i] \setminus \cup_{j=1}^{i-1}N_G[x_j]\not=\emptyset$, i.e., $x_i$ is a legal choice also in $G$.  Now, suppose that $x_i\in N_{G-e}[u]$. If $f_{S}^{-1}(x_i)\neq \{u\}$, then again $x_i$ is a legal choice also in $G$. But even if $f_{S}^{-1}(x_i)= \{u\}$, and $v\notin\{x_1,\ldots,x_{i-1}\}$, $x_i$ is a legal choice in $G$. Thus the only problem with legality of $x_i$ is when $v=x_j$ for some $j<i$ (note that we are assuming $x_i\in N_{G-e}[u]$ and $f_{S}^{-1}(x_i)= \{u\}$, that is, the only vertex in $G-e$ footprinted by $x_i$ is $u$). Now, let $S'$ be the sequence obtained from $S$ by removing $x_i$, i.e., $S'=(x_1,\ldots,x_{i-1},x_{i+1},\ldots x_k)$. From the above it is clear that $S'$ is a legal sequence in $G$ of length $\ggr(G-e)-1$, hence $\ggr(G)\ge |\widehat{S'}| =\ggr(G-e)-1$.

To see the second part of the theorem, consider the family of graphs $H_{m,n}$, obtained by identifying a vertex of degree 1 of the path $P_m$, $m\ge 3$, with a vertex of the cycle $C_n$, $n\ge 3$. (The meaning of {\em edges of the cycle} and {\em edges of the path} in $H_{m,n}$ should be clear.) It is easy to see that $\ggr(H_{m,n})=m+n-3$. Now, if $e$ is an edge of the cycle in $H_{m,n}$, then $\ggr(G_{m,n}-e)=m+n-2$. If $e$ is the pendant edge of the path in $H_{m,n}$, then $\ggr(G_{m,n}-e)=m+n-3=\ggr(G_{m,n})$. Finally, if $e$ is any other edge (of the path), then $\ggr(G_{m,n}-e)=m+n-4$. 
\qed
\end{proof}

The following immediate consequence of Theorem~\ref{thm:edge-removed} 
will be used later. 

\begin{corollary}
\label{cor:k-edges}
Let $G$ be a graph, and let $G'$ be obtained from $G$, by adding $k$ edges to $G$. 
Then $\ggr(G)-k\le \ggr(G') \le \ggr(G)+k$.
\end{corollary}

\subsection{Vertex-deletion}

It is easy to see that $\ggr(H)\le \ggr(G)$, if $H$ is an induced subgraph of $G$. Indeed, if $S$ is a Grundy dominating sequence of $H$,
then $S$ is also a sequence in $G$, and it is clearly legal also with respect to $G$. If $S$ is not a dominating sequence in $G$, one can add 
some vertices in a legal way at the end of $S$ to make it a dominating sequence of $G$. Hence $\ggr(G)\ge \ggr(H)$, and so the graph property 
of having the Grundy domination number bounded from above by a constant belongs to hereditary properties.

Let us focus on the action of vertex deletion in a graph $G$. By the observation in the previous paragraph, the Grundy domination number cannot increase when a vertex is removed. The following result specifies how much it can decrease.

\begin{theorem}
\label{thm:vertex-removed}
If $G$ is a graph and $u\in V(G)$, then $$\ggr(G)-2\le \ggr(G-u) \le \ggr(G).$$
\noindent Moreover, there exist graphs $G$ such that all values of $\ggr(G-u)$ between 
$\ggr(G)-2$ and $\ggr(G)$ are realized for different vertices $u\in V(G)$.
\end{theorem}
\begin{proof}
The bound $\ggr(G-u) \le \ggr(G)$ immediately follows from the fact that the Grundy domination number of an induced subgraph $H$ of $G$
is not greater than that of $G$. 

For bounding $\ggr(G-u)$ from below, let $S$ be a Grundy dominating sequence in $G$, and let $v$ be the vertex in $S$ that footprints $u$, i.e., $v=f_S(u)$. Consider the sequence $S'$ obtained from $S$ by removing $v$, and, if $u\in \widehat{S}$, also removing $u$. 
We claim that $S'$ is a legal sequence in $G-u$. Indeed, since $v$ and $u$ are not in $S'$, we derive that each vertex $x$ from $S'$ in $G-u$ footprints all the vertices that are footprinted by $x$ with respect to $S$ in $G$ (while $x$ with respect to $S'$ in $G-u$ could also footprint some additional vertices in $N_G(u)\cup N_G[v]$). Since a legal sequence $S'$ can be completed to a dominating sequence of $G-u$, we get $\ggr(G)-2\le |\widehat{S'}| \le \ggr(G-u)$.

To see the second part of the theorem, consider the family of graphs $G_{m,n}$, obtained by identifying a vertex of degree 1 of the path $P_m$, $m\ge 4$, with a vertex of the complete graph $K_n$, $n\ge 3$. It is clear that $\ggr(G_{m,n})=m$. Now, if $u$ is the vertex of degree $1$ in $G_{m,n}$ or its neighbor or the identified vertex, then $\ggr(G_{m,n}-u)=m-1$. If $u$ is a vertex of the complete graph (and not the identified vertex), then $\ggr(G_{m,n}-u)=m$. Finally, if $u$ is any other vertex (in the path), then $\ggr(G_{m,n}-u)=m-2$.
\qed
\end{proof}
 
Better bounds can be obtained for special type of vertices, namely for twin and simplicial vertices, whose definition we recall now. A vertex $v\in V(G)$ is a {\emph{simplicial}} vertex of a graph $G$, if $N(v)$ induces a complete graph. 
Two vertices $u$ and $v$ in $G$ are called {\emph{twins}} if $N[u]=N[v]$. A vertex $v\in V$ is called a {\emph{twin vertex}} if there exists $u\in V$, such that $u$ and $v$ are twins.   
The following result will be applied later (the second statement was known already in~\cite{bgmrr-2014}).

\begin{proposition}\label{prp:simtwin}
Let $G$ be a graph and $u\in V(G)$. 
\begin{enumerate}[(i)]
\item
If $u$ is a simplicial vertex, then $\ggr(G-u)\geq \ggr(G)-1$. 
\item
If $u$ is a twin vertex, then $\ggr(G-u)=\ggr(G)$.
\end{enumerate}
\end{proposition}
\begin{proof}
\textit{(i)} Let $S$ be a Grundy dominating sequence in $G$, $u$ a simplicial vertex, and let $f_S(u)=v$. If $u \notin \widehat{S}$, then the sequence obtained from $S$ by removing $v$ is a legal sequence in $G-u$, implying $\ggr(G-u) \geq \ggr(G)-1.$ Suppose now that $u \in \widehat{S}.$ Then $u$ footprints one vertex from the clique $N[u]$. Thus $u$ is in $S$ before any $x \in N(u) \cap \widehat{S}$ which means that $f_S(u)=u.$ Thus the sequence obtained from $S$ by removing $u$ is a legal sequence in $G-u$, implying $\ggr(G-u) \geq \ggr(G)-1.$  

\textit{(ii)} Let $v$ be a twin of $u$ and let $S$ be a Grundy dominating sequence in $G$. If $u \notin \widehat{S}$, then $S$ is a legal sequence of $G-u$, implying $\ggr(G-u) \geq \ggr(G).$ Suppose now that $u \in \widehat{S}.$ Then $v\notin \widehat{S}$ and the sequence $S'$ obtained from $S$ by replacing $u$ with $v$ is a Grundy dominating sequence in $G$ not containing $u$. Hence also in this case we have $\ggr(G-u) \geq \ggr(G).$ Combining this with Theorem~\ref{thm:vertex-removed} we obtain $\ggr(G-u)=\ggr(G).$  
\qed
\end{proof}

%%%%%%%%%%%%%%%%%%%%%%%%%%%%%%%%%%%%%%%%%%%%%%%%%%%%%%%%%%%%%%%%
%%%%%%%%%%%%%%%%%%%%%%%%%%%%%%%%%%%%%%%%%%%%%%%%%%%%%%%%%%%%%%%%
\section{Dominating sequences of Sierpi\'nski graphs}
\label{s:Sierpinski}
%%%%%%%%%%%%%%%%%%%%%%%%%%%%%%%%%%%%%%%%%%%%%%%%%%%%%%%%%%%%%%%%
%%%%%%%%%%%%%%%%%%%%%%%%%%%%%%%%%%%%%%%%%%%%%%%%%%%%%%%%%%%%%%%%
%povzeto po Crossing Numbers of Sierpi\'nski-Like Graphs
%in po Packing chromatic number of base-3 Sierpinski graphs
Set $[n] = \{1,2,\dots,n \}$ and $[n]_0 = \{0,1,\dots,n-1 \}$.
The Sierpi\'nski graph $S_p^n$ ($n,p \ge 1$) is defined on the vertex set $[p]_0^n$,
two different vertices $u = (u_1,u_2,\dots,u_n)$ and $v = (v_1,v_2,\dots,v_n)$ being adjacent if and
only if there exists an $h \in [n]$ such that
\begin{enumerate}
	\item $u_t = v_t$, for $t = 1,2,\dots,h-1$;
	\item $u_h \neq v_h$; and
	\item $u_t = v_h$ and $u_h = v_t$ for $t = h+1,h+2,\dots,n$;
\end{enumerate}
In the rest, we will shortly write $\langle u_1u_2\dots u_{n}\rangle$ for $(u_1,u_2,\dots,u_n)$ and $u_j$ will be called {\em $j$-th bit} of a vertex $u$.
A vertex of the form $\langle ii\dots i\rangle = \langle i^n \rangle$ of $S_p^n$ is called an {\em extreme vertex}. The extreme vertices of $S_p^n$ are of degree $p-1$ while the degree of any other vertex is $p$.
%The graphs $S(n,k)$ are defined as follows.  $S(0,k) = K_1$ (so that $E(S(0,k)) = \emptyset$).
%For $n\ge 1$, the vertex set of $S(n,k)n$ is $[k]_0^n$ and the edge set is defined recursively as
%$$E(S(n,k)) = \{ \{is,it\}:\ i\in [k]_0\,, \{s,t\}\in E(S(n-1,k))\} \cup
%\{ \{ij^{n-1}, ji^{n-1}\} \ |\ i,j\in [k]_0\,, i\ne j \}\,.
%$$

In other words, $S_p^n$ can be constructed from $p$ copies of $S_p^{n-1}$ as follows. For each $i\in [p]_0$ concatenate $i$ to the left of the vertices in a copy of $S_p^{n-1}$ and denote the obtained graph with $iS_p^{n-1}$. Then for each $i\neq j$ join copies $iS_p^{n-1}$ and $jS_p^{n-1}$ by the single edge $e_{ij}^{(n)} = \{ij^{n-1}, ji^{n-1}\}$.
In Fig.~\ref{fig:S3} the construction of $S_3^{3}$ is illustrated.

\begin{figure}[ht!]
	\begin{center}
		\begin{tikzpicture}[scale=1.0,style=thick,x=1cm,y=1cm]
		\def\vr{2.5pt} % \vr = vertex radius;
		%  edges e_{ij}^{(3)}
		\draw (8,9)--(7.4,8);
		\draw (9.2,5)--(10.4,5);
		\draw (11.6,9)--(12.2,8);
		% 0 S_3^2
		\draw (9.8,12)--(9.2,11)--(10.4,11)--cycle;
		\draw (8.6,10)--(8,9)--(9.2,9)--cycle;
		\draw (11,10)--(10.4,9)--(11.6,9)--cycle;
		\draw (9.2,11)--(8.6,10);
		\draw (9.2,9)--(10.4,9);
		\draw (10.4,11)--(11,10);
		\draw (9.8,12) [fill=white] circle (\vr);
		\draw (9.2,11) [fill=white] circle (\vr);
		\draw (10.4,11) [fill=white] circle (\vr);
		\draw (8.6,10) [fill=white] circle (\vr);
		\draw (8,9) [fill=white] circle (\vr);
		\draw (9.2,9) [fill=white] circle (\vr);
		\draw (11,10) [fill=white] circle (\vr);
		\draw (10.4,9) [fill=white] circle (\vr);
		\draw (11.6,9) [fill=white] circle (\vr);
		\draw[anchor = south] (9.8,12) node {{\bf 0}00};
		\draw[anchor = east] (9.2,11) node {{\bf 0}01};
		\draw[anchor = west] (10.4,11) node {{\bf 0}02};
		\draw[anchor = east] (8.6,10) node {{\bf 0}10};
		\draw[anchor = east] (8,9) node {{\bf 0}11};
		\draw[anchor = north] (9.2,9) node {{\bf 0}12};
		\draw[anchor = west] (11,10) node {{\bf 0}20};
		\draw[anchor = north] (10.4,9) node {{\bf 0}21};
		\draw[anchor = west] (11.6,9) node {{\bf 0}22};
		% 1 S_3^2
		\draw (7.4,8)--(6.8,7)--(8,7)--cycle;
		\draw (6.2,6)--(5.6,5)--(6.8,5)--cycle;
		\draw (8.6,6)--(8,5)--(9.2,5)--cycle;
		\draw (6.8,7)--(6.2,6);
		\draw (6.8,5)--(8,5);
		\draw (8,7)--(8.6,6);
		\draw (7.4,8) [fill=white] circle (\vr);
		\draw (6.8,7) [fill=white] circle (\vr);
		\draw (8,7) [fill=white] circle (\vr);
		\draw (6.2,6) [fill=white] circle (\vr);
		\draw (5.6,5) [fill=white] circle (\vr);
		\draw (6.8,5) [fill=white] circle (\vr);
		\draw (8.6,6) [fill=white] circle (\vr);
		\draw (8,5) [fill=white] circle (\vr);
		\draw (9.2,5) [fill=white] circle (\vr);
		\draw[anchor = east] (7.4,8) node {100};
		\draw[anchor = east] (6.8,7) node {101};
		\draw[anchor = west] (8,7) node {102};
		\draw[anchor = east] (6.2,6) node {110};
		\draw[anchor = north east] (5.6,5) node {111};
		\draw[anchor = north] (6.8,5) node {112};
		\draw[anchor = west] (8.6,6) node {120};
		\draw[anchor = north] (8,5) node {121};
		\draw[anchor = north] (9.2,5) node {122};
		% 2 S_3^2
		\draw (12.2,8)--(11.6,7)--(12.8,7)--cycle;
		\draw (11,6)--(10.4,5)--(11.6,5)--cycle;
		\draw (13.4,6)--(12.8,5)--(14,5)--cycle;
		\draw (11.6,7)--(11,6);
		\draw (11.6,5)--(12.8,5);
		\draw (12.8,7)--(13.4,6);
		\draw (12.2,8) [fill=white] circle (\vr);
		\draw (11.6,7) [fill=white] circle (\vr);
		\draw (12.8,7) [fill=white] circle (\vr);
		\draw (11,6) [fill=white] circle (\vr);
		\draw (10.4,5) [fill=white] circle (\vr);
		\draw (11.6,5) [fill=white] circle (\vr);
		\draw (13.4,6) [fill=white] circle (\vr);
		\draw (12.8,5) [fill=white] circle (\vr);
		\draw (14,5) [fill=white] circle (\vr);
		\draw[anchor = west] (12.2,8) node {200};
		\draw[anchor = east] (11.6,7) node {201};
		\draw[anchor = west] (12.8,7) node {202};
		\draw[anchor = east] (11,6) node {210};
		\draw[anchor = north] (10.4,5) node {211};
		\draw[anchor = north] (11.6,5) node {212};
		\draw[anchor = west] (13.4,6) node {220};
		\draw[anchor = north] (12.8,5) node {221};
		\draw[anchor = north west] (14,5) node {222};
		\draw (7.95,8.25) node {$e_{01}^{(3)}$};
		\draw (11.60,8.25) node {$e_{02}^{(3)}$};
		\draw (9.8,5.35) node {$e_{12}^{(3)}$};
		\draw (7.0,8.55)--(12.5,8.55)--(12.5,12.5)--(7.0,12.5)--cycle;
		\draw (12.0,12.2) node {${\bf 0}S_3^2$};
		\end{tikzpicture}
	\end{center}
	\caption{The Sierpi\'nski graph $S_3^{3}$}
	\label{fig:S3}
\end{figure}
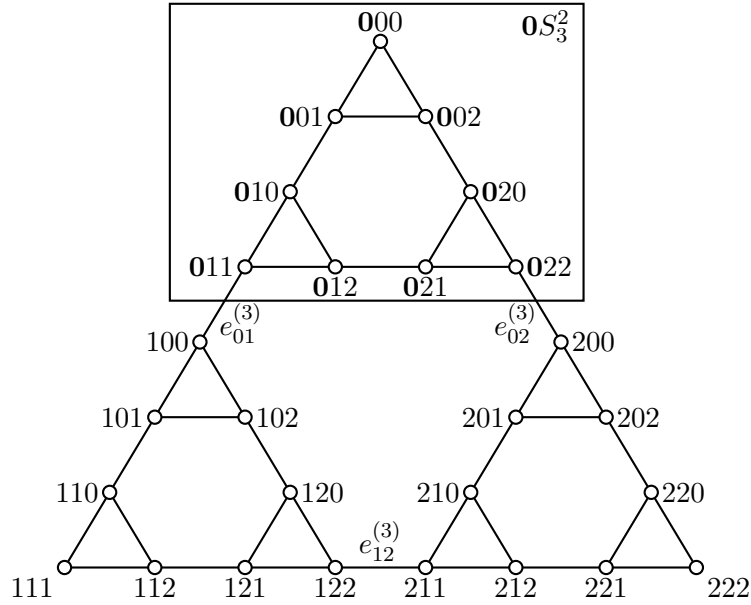

A dominating sequence of $S_p^{n}$ built in the proof of the following result will be denoted by $A_p^{n}$. Note that we can construct a dominating sequence of $iS_p^{n}$ by concatenating $i$ to the left of every vertex in the sequence $A_p^{n}$, and we denote this sequence by $iA_p^{n}$.

\begin{theorem}
	\label{thm:Sier}
	If $n,p\ge 1$ and $S_p^{n}$ is a Sierpi\'nski graph, then
	\begin{equation}
	 \gamma_{gr}(S_p^{n}) = p^{n-1} + \frac{p(p^{n-1}-1)}{2}.
	 \end{equation}
\end{theorem}
\begin{proof}
	Fix $p \ge 1$. It is clear that  $\gamma_{gr}(S_p^{1}) = \gamma_{gr}(K_p) = 1$. Let $G$ be a graph of $p$ disjoint copies of $S_p^{n-1}$, where $n \ge 2$. It is easy to see that $ \gamma_{gr}(G) = p \cdot \gamma_{gr}(S_p^{n-1})$. To construct $S_p^{n}$ from $G$, ${p}\choose{2}$ edges have to be added. Hence by Corollary \ref{cor:k-edges}, if ${p}\choose{2}$ edges are added, the Grundy domination number  increases by at most ${p}\choose{2}$. It follows that $\gamma_{gr}(S_p^{n}) \le \gamma_{gr}(G) + {{p}\choose{2}} =  p \cdot \gamma_{gr}(S_p^{n-1}) + {{p}\choose{2}} $.
Note that the recursion $a_n = p a_{n-1} + \binom{p}{2}$ with base $a_1 = 1$ can be converted to the explicit form $a_n = p^{n-1} + \frac{p(p^{n-1}-1)}{2}$.
Thus we derive the upper bound
	\begin{equation}
		\label{prf:Sier1}
		\gamma_{gr}(S_p^{n}) \le p^{n-1} + \frac{p(p^{n-1}-1)}{2}.
	\end{equation}
	
	For the reversed inequality, we will construct a legal dominating sequence in $S_p^{n}$ of length $p \cdot \gamma_{gr}(S_p^{n-1}) + {{p}\choose{2}} $, which will be denoted by $A_p^{n}$. First, we define another sequence $iB_p^{n} = (\langle i(i+1)^{n-1}\rangle,\langle i(i+2)^{n-1}\rangle, \dots ,\langle i(p-2)^{n-1}\rangle, \langle i(p-1)^{n-1}\rangle)$, where $i \in [p]_0$. So in $iB_p^{n}$ are just the extreme vertices of $iS_p^{n-1}$, where $i$ is smaller then all other bits. The sequence $A_p^{n}$ will be constructed recursively. For the recursion base the dominating sequence of $S_p^{1}$ is $A_p^{1} = (\langle 0\rangle)$.
	The sequence $A_p^{n}$, for $n>1$ is constructed as follows:
	\begin{eqnarray*}
		A_p^{n}  = & 0A_p^{n-1}\ \oplus\ 0B_p^{n} \ \oplus\  1A_p^{n-1} \ \oplus\  1B_p^{n} \ \oplus\  \dots \\
		 & \oplus\  (p-2)A_p^{n-1} \ \oplus\  (p-1)B_p^{n} \ \oplus\  (p-1)A_p^{n-1}
	\end{eqnarray*}
	We can also write
	$$	
		A_p^{n} = \bigoplus\limits_{i=0}^{p-1} \bigl( iA_p^{n-1}\ \oplus\ iB_p^{n} \bigr),
	$$
	where $(p-1)B_p^{n}$ is an empty sequence. In Fig.~\ref{fig:S3dom} the  $\ggr$-sequences of $S_3^{1}$, $S_3^{2}$ and $S_3^{3}$ are illustrated.
	
	Clearly $\langle 0^n\rangle $ is the first vertex in $A_p^{n}$. Now, we show that $\langle 0^n\rangle$ is the only extreme vertex in the sequence. To see that, we have to expand the equation to the bottom of recursion. Then each vertex in the sequence comes from some $iA_p^{1}$ or some $iB_p^{l}$, where $i \in [p]_0$ and $1 < l \leq n$. If a vertex comes from $iA_p^{1}$, then its last bit is $0$. The only extreme vertex with the last bit $0$ is $\langle 0^n\rangle $. If a vertex is in some $iB_p^{l}$, then it is not an extreme vertex in $A_p^{n}$, since its $l$-th bit is $i$ and its last $(l-1)$ bits are greater then $i$.
So the only extreme vertex in $A_p^{n}$ is $\langle 0^n\rangle $.
	
	It is easy to see, that vertices are pairwise different. If $i,j \in [p]_0$ and $i \neq j$, then vertices in $iA_p^{n-1}$ and vertices in $jA_p^{n-1}$ differ already in the first bit. The same holds for vertices in $iB_p^{n}$ and $jB_p^{n}$ and vertices in $iA_p^{n-1}$ and $jB_p^{n}$.
	In some $iB_p^{n}$ are just the vertices that are extreme vertices in $S_p^{n-1}$, and in $iA_p^{n-1}$ the only extreme vertex is $\langle  i0^{n-1}\rangle$. But vertex $\langle  i0^{n-1}\rangle$ is not in $iB_p^{n}$ since $0 \leq i$.
	So, vertices that are in $iA_p^{n-1}$ are not in $iB_p^{n}$.
	
	To show that the sequence $A_p^{n}$ is legal, we will check that every vertex in the sequence is footprinting at least one vertex.
	We mentioned already that every vertex in $A_p^{n}$ comes either from $iA_p^{1}$ or some $iB_p^{l}$, where $i \in [p]_0$ and $1 < l \leq n$.
	If the vertex $v$ comes from a sequence $iB_p^{l}$, its form is $v =\langle x_1x_2\dots x_{n-l}ab^{l-1}\rangle$, where $a < b$ and $x_1,\ldots,x_{n-l},a,b \in [p]_0$. The vertex $v$ footprints $u = \langle x_1x_2\dots x_{n-l}ba^{l-1}\rangle$, because $a<b$, and $u$ and all its other neighbors cannot be in $A_p^{n}$ before $v$.
	
	If $v$ comes from some $iA_p^{1}$, its form is $v = \langle x_1x_2\dots x_{n-2}i0\rangle$, and we claim that $v$ footprints at least the vertex $u=\langle x_1x_2\dots x_{n-2}i(p-1)\rangle$. Clearly, $u$ is not footprinted before by some other vertex from $jA_p^{n}$. Vertices that are in some $jB_p^{l}$ and are also in $N[u]$ are just some vertices of $iB_p^{2}$. Since $iB_p^{2}$ is in $A_p^{n}$ after $iA_p^{1}$, it follows that $v$ footprints $u$.
	
	It is obvious that the vertices of $A_p^{n}$ dominate the whole graph, since already the vertices of $\bigoplus\limits_{i=0}^{p-1}  iA_p^{n-1}$ dominates it.
	So $A_p^{n}$ is a (legal) dominating sequence and its length is
	 \begin{eqnarray*}
	 	|A_p^{n}| & = \sum\limits_{i=0}^{p-1} \bigl( |iA_p^{n-1}| + |iB_p^{n}| \bigr) \\
	 	& = \sum\limits_{i=0}^{p-1} \bigl( |A_p^{n-1}| + (p-1-i) \bigr) \\
	 	& = p \cdot |A_p^{n-1}| + \binom{p}{2}.
	 \end{eqnarray*}
	Since $|A_p^{1}|=1$, we can transform the recursion to the explicit form
	$$ |A_p^{n}| = p^{n-1} + \frac{p(p^{n-1}-1)}{2}.$$
	It follows
	\begin{equation}
		\label{prf:Sier2}
		\gamma_{gr}(S_p^{n}) \ge |A_p^{n}| = p^{n-1} + \frac{p(p^{n-1}-1)}{2}.
	\end{equation}
	Combining inequalities \eqref{prf:Sier1} and \eqref{prf:Sier2}, we get
	$$\gamma_{gr}(S_p^{n}) = p^{n-1} + \frac{p(p^{n-1}-1)}{2}.$$

	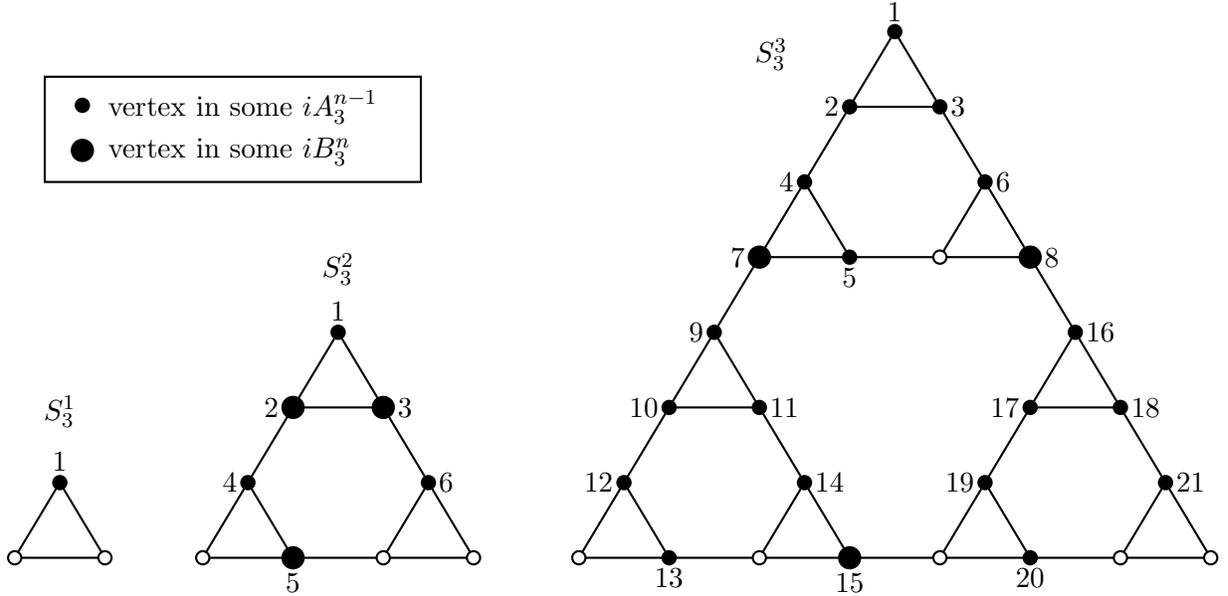
\begin{figure}[ht!]
		\begin{center}
			\begin{tikzpicture}[scale=1.0,style=thick,x=1cm,y=1cm]
			\def\vr{2.5pt} % \vr = vertex radius;
			\def\bvr{4.0pt}
			% S_1
			\draw[anchor = north] (-1.3,7.3) node {$S_3^{1}$};
			\draw (-1.3,6)--(-1.9,5)--(-0.7,5)--cycle;
			\draw (-1.3,6) [fill=black] circle (\vr);
			\draw (-1.9,5) [fill=white] circle (\vr);
			\draw (-0.7,5) [fill=white] circle (\vr);
			\draw[anchor = north] (-1.3,6.55) node {1};
			% S_2
			\draw[anchor = north] (2.4,9.2) node {$S_3^{2}$};
			\draw (2.4,8)--(1.8,7)--(3,7)--cycle;
			\draw (1.2,6)--(0.6,5)--(1.8,5)--cycle;
			\draw (3.6,6)--(3,5)--(4.2,5)--cycle;
			\draw (1.8,7)--(1.2,6);
			\draw (1.8,5)--(3,5);
			\draw (3,7)--(3.6,6);
			\draw (2.4,8) [fill=black] circle (\vr);
			\draw (1.8,7) [fill=black] circle (\bvr);
			\draw (3,7) [fill=black] circle (\bvr);
			\draw (1.2,6) [fill=black] circle (\vr);
			\draw (0.6,5) [fill=white] circle (\vr);
			\draw (1.8,5) [fill=black] circle (\bvr);
			\draw (3.6,6) [fill=black] circle (\vr);
			\draw (3,5) [fill=white] circle (\vr);
			\draw (4.2,5) [fill=white] circle (\vr);
			\draw[anchor = north] (2.4,8.55) node {1};
			\draw[anchor = east] (1.75,7) node {2};
			\draw[anchor = west] (3.05,7) node {3};
			\draw[anchor = east] (1.2,6) node {4};
			%	\draw[anchor = north east] (5.6,5) node {111};
			\draw[anchor = north] (1.8,4.93) node {5};
			\draw[anchor = west] (3.6,6) node {6};
			%	\draw[anchor = north] (8,5) node {121};
		%	\draw[anchor = north] (4.2,4.93) node {15};
			% S_3
			\draw[anchor = west] (7.8,11.7) node {$S_3^{3}$};
			%  edges e_{ij}^{(3)}
			\draw (8,9)--(7.4,8);
			\draw (9.2,5)--(10.4,5);
			\draw (11.6,9)--(12.2,8);
			% 0 S_3^2
			\draw (9.8,12)--(9.2,11)--(10.4,11)--cycle;
			\draw (8.6,10)--(8,9)--(9.2,9)--cycle;
			\draw (11,10)--(10.4,9)--(11.6,9)--cycle;
			\draw (9.2,11)--(8.6,10);
			\draw (9.2,9)--(10.4,9);
			\draw (10.4,11)--(11,10);
			\draw (9.8,12) [fill=black] circle (\vr);
			\draw (9.2,11) [fill=black] circle (\vr);
			\draw (10.4,11) [fill=black] circle (\vr);
			\draw (8.6,10) [fill=black] circle (\vr);
			\draw (8,9) [fill=black] circle (\bvr);
			\draw (9.2,9) [fill=black] circle (\vr);
			\draw (11,10) [fill=black] circle (\vr);
			\draw (10.4,9) [fill=white] circle (\vr);
			\draw (11.6,9) [fill=black] circle (\bvr);
			\draw[anchor = south] (9.8,12) node {1};
			\draw[anchor = east] (9.2,11) node {2};
			\draw[anchor = west] (10.4,11) node {3};
			\draw[anchor = east] (8.6,10) node {4};
			\draw[anchor = east] (7.95,9) node {7};
			\draw[anchor = north] (9.2,9) node {5};
			\draw[anchor = west] (11,10) node {6};
		%	\draw[anchor = north] (10.4,9) node {{\bf 0}21};
			\draw[anchor = west] (11.65,9) node {8};
			% 1 S_3^2
			\draw (7.4,8)--(6.8,7)--(8,7)--cycle;
			\draw (6.2,6)--(5.6,5)--(6.8,5)--cycle;
			\draw (8.6,6)--(8,5)--(9.2,5)--cycle;
			\draw (6.8,7)--(6.2,6);
			\draw (6.8,5)--(8,5);
			\draw (8,7)--(8.6,6);
			\draw (7.4,8) [fill=black] circle (\vr);
			\draw (6.8,7) [fill=black] circle (\vr);
			\draw (8,7) [fill=black] circle (\vr);
			\draw (6.2,6) [fill=black] circle (\vr);
			\draw (5.6,5) [fill=white] circle (\vr);
			\draw (6.8,5) [fill=black] circle (\vr);
			\draw (8.6,6) [fill=black] circle (\vr);
			\draw (8,5) [fill=white] circle (\vr);
			\draw (9.2,5) [fill=black] circle (\bvr);
			\draw[anchor = east] (7.4,8) node {9};
			\draw[anchor = east] (6.8,7) node {10};
			\draw[anchor = west] (8,7) node {11};
			\draw[anchor = east] (6.2,6) node {12};
		%	\draw[anchor = north east] (5.6,5) node {111};
			\draw[anchor = north] (6.8,5) node {13};
			\draw[anchor = west] (8.6,6) node {14};
		%	\draw[anchor = north] (8,5) node {121};
			\draw[anchor = north] (9.2,4.93) node {15};
			% 2 S_3^2
			\draw (12.2,8)--(11.6,7)--(12.8,7)--cycle;
			\draw (11,6)--(10.4,5)--(11.6,5)--cycle;
			\draw (13.4,6)--(12.8,5)--(14,5)--cycle;
			\draw (11.6,7)--(11,6);
			\draw (11.6,5)--(12.8,5);
			\draw (12.8,7)--(13.4,6);
			\draw (12.2,8) [fill=black] circle (\vr);
			\draw (11.6,7) [fill=black] circle (\vr);
			\draw (12.8,7) [fill=black] circle (\vr);
			\draw (11,6) [fill=black] circle (\vr);
			\draw (10.4,5) [fill=white] circle (\vr);
			\draw (11.6,5) [fill=black] circle (\vr);
			\draw (13.4,6) [fill=black] circle (\vr);
			\draw (12.8,5) [fill=white] circle (\vr);
			\draw (14,5) [fill=white] circle (\vr);
			\draw[anchor = west] (12.2,8) node {16};
			\draw[anchor = east] (11.6,7) node {17};
			\draw[anchor = west] (12.8,7) node {18};
			\draw[anchor = east] (11,6) node {19};
		%	\draw[anchor = north] (10.4,5) node {211};
			\draw[anchor = north] (11.6,5) node {20};
			\draw[anchor = west] (13.4,6) node {21};
		%	\draw[anchor = north] (12.8,5) node {221};
		%	\draw[anchor = north west] (14,5) node {222};
			%
			\draw (-1.5,11.4)--(-1.5,10)--(3.5,10)--(3.5,11.4)--cycle;
			\draw (-1.0,11.02) [fill=black] circle (\vr);
			\draw[anchor = west] (-0.8,11.0) node {vertex in some $iA_3^{n-1}$};
			\draw (-1.0,10.42) [fill=black] circle (\bvr);
			\draw[anchor = west] (-0.8,10.4) node {vertex in some $iB_3^{n}$};
			\end{tikzpicture}
		\end{center}
		\caption{The Sierpi\'nski graphs $S_3^{1}$, $S_3^{2}$ and $S_3^{3}$ with $\ggr$-sequences.}
		\label{fig:S3dom}
	\end{figure}
	
	\qed
\end{proof}

%The proof of Theorem~\ref{thm:Sier} also yields a construction of the Grundy dominating sequence of $S_p^{n}$., which is of time complexity $O(np^n)$.

The proof of Theorem~\ref{thm:Sier} also yields the construction of a Grundy dominating sequence of $S_p^{n}$. In the construction we are building the Grundy dominating sequence by concatenating smaller sequences and by concatenating one bit to the left of every vertex in the sequence. Thus the time complexity of the construction is the same as the length of the Grundy dominating sequence multiplied by $n$, since each vertex is labelled by $n$ bits. We derive that the time complexity of the construction is $O(np^n)$. In addition, the algorithm constructs along the way all Grundy dominating sequences of $S_p^{\ell}$, for $\ell\le n$. Since the complexity of the algorithm (which simply builds the sequence $A_p^{n}$) is the same as generating the labels of vertices in $S_p^{n}$ that form a Grundy dominating sequence, we infer that this complexity is best possible. 

\begin{corollary}
The time complexity of constructing a Grundy dominating sequence ($A_p^{n}$) of the Sierpi\' nski graph $S_p^{n}$ is $O(np^n)$, and this is best possible.
\end{corollary}

There exists an alternative (and easier) construction of a $\gamma_{gr}$-sequence of $S_p^{n}$, of which construction time complexity is slightly worse than above, since we are going through all vertices in $S_p^{n}$. Nevertheless, let us present also this construction, because the vertices in the sequence are nicely listed in the lexicographical order.
This means that we order the vertices lexicographically by their labels and in this order we add to the sequence each legal vertex (a vertex whose neighborhood is not contained in the union of neighborhoods of previously chosen vertices).   
 We denote the sequence by $L_p^{n}$. A vertex $v$ is in $L_p^{n}$ if and only if
\begin{itemize}
		\item the last bit of $v$ is $0$ or
		\item $v = \langle x_1x_2\dots x_{n-l}ab^{l-1}\rangle$, where $2 \le l \le n$, $b > a$ and $x_1,\ldots,x_{n-l},a,b \in [p]_0$.
\end{itemize}
	
Let us show that $L_p^{n}$ is legal. If the last bit of $v$ is $0$ ($v = \langle x_1x_2\dots x_{n-1}0\rangle$) then $v$ footprints at least the vertex $u = \langle x_1x_2\dots x_{n-1}(p-1)\rangle$. Note that $u$ is not footprinted by any other vertex of $L_p^{n}$ since all its other neighbors are lexicographically greater than $v$. So if they are in $L_p^{n}$ then they are in $L_p^{n}$ after $v$.
If $v = \langle x_1x_2\dots x_{n-l}ab^{l-1}\rangle$, where $2 \le l \le n$, $b > a$ and $x_1,\ldots,x_{n-l},a,b \in [p]_0$ then $v$ footprints $u = \langle x_1x_2\dots x_{n-l}ba^{l-1}\rangle$. Since $b > a$, all other neighbors of $u$ are also lexicographically greater then $v$ (note that they are of the form $\langle x_1x_2\dots x_{n-l}ba^{l-2}y\rangle$, where $y \in [p]_0\setminus\{a\}$). It is clear that the vertices of $L_p^{n}$ dominate the whole graph, since already the vertices with the last bit $0$ dominate it. So $L_p^{n}$ is a (legal) dominating sequence.
	
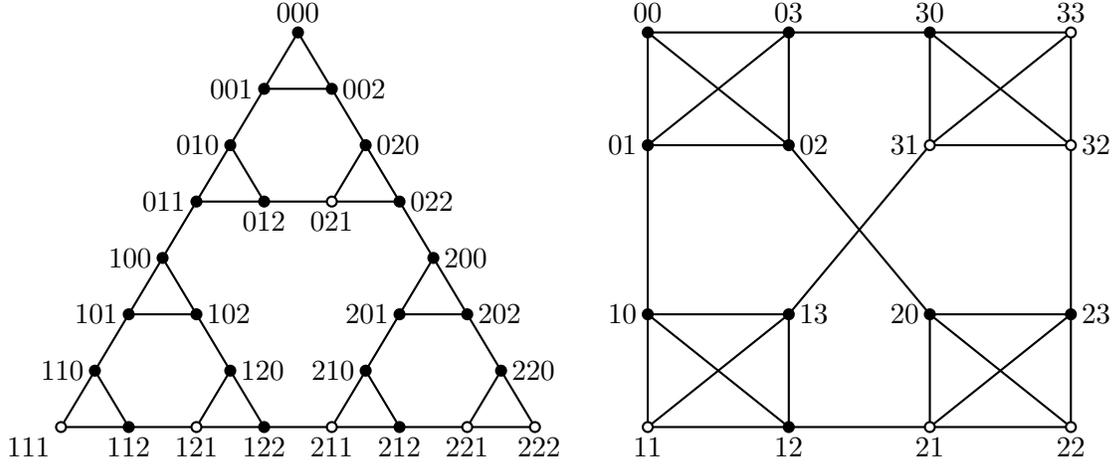
\begin{figure}[ht!]
		
\begin{center}
			\begin{tikzpicture}[scale=0.75,style=thick,x=1cm,y=1cm]
			\def\vr{2.5pt} % \vr = vertex radius;
			%  edges e_{ij}^{(3)}
			\draw (8,9)--(7.4,8);
			\draw (9.2,5)--(10.4,5);
			\draw (11.6,9)--(12.2,8);
			% 0 S_3^2
			\draw (9.8,12)--(9.2,11)--(10.4,11)--cycle;
			\draw (8.6,10)--(8,9)--(9.2,9)--cycle;
			\draw (11,10)--(10.4,9)--(11.6,9)--cycle;
			\draw (9.2,11)--(8.6,10);
			\draw (9.2,9)--(10.4,9);
			\draw (10.4,11)--(11,10);
			\draw (9.8,12) [fill=black] circle (\vr);
			\draw (9.2,11) [fill=black] circle (\vr);
			\draw (10.4,11) [fill=black] circle (\vr);
			\draw (8.6,10) [fill=black] circle (\vr);
			\draw (8,9) [fill=black] circle (\vr);
			\draw (9.2,9) [fill=black] circle (\vr);
			\draw (11,10) [fill=black] circle (\vr);
			\draw (10.4,9) [fill=white] circle (\vr);
			\draw (11.6,9) [fill=black] circle (\vr);
			\draw[anchor = south] (9.8,12) node {000};
			\draw[anchor = east] (9.2,11) node {001};
			\draw[anchor = west] (10.4,11) node {002};
			\draw[anchor = east] (8.6,10) node {010};
			\draw[anchor = east] (8,9) node {011};
			\draw[anchor = north] (9.2,9) node {012};
			\draw[anchor = west] (11,10) node {020};
			\draw[anchor = north] (10.4,9) node {021};
			\draw[anchor = west] (11.6,9) node {022};
			% 1 S_3^2
			\draw (7.4,8)--(6.8,7)--(8,7)--cycle;
			\draw (6.2,6)--(5.6,5)--(6.8,5)--cycle;
			\draw (8.6,6)--(8,5)--(9.2,5)--cycle;
			\draw (6.8,7)--(6.2,6);
			\draw (6.8,5)--(8,5);
			\draw (8,7)--(8.6,6);
			\draw (7.4,8) [fill=black] circle (\vr);
			\draw (6.8,7) [fill=black] circle (\vr);
			\draw (8,7) [fill=black] circle (\vr);
			\draw (6.2,6) [fill=black] circle (\vr);
			\draw (5.6,5) [fill=white] circle (\vr);
			\draw (6.8,5) [fill=black] circle (\vr);
			\draw (8.6,6) [fill=black] circle (\vr);
			\draw (8,5) [fill=white] circle (\vr);
			\draw (9.2,5) [fill=black] circle (\vr);
			\draw[anchor = east] (7.4,8) node {100};
			\draw[anchor = east] (6.8,7) node {101};
			\draw[anchor = west] (8,7) node {102};
			\draw[anchor = east] (6.2,6) node {110};
			\draw[anchor = north east] (5.6,5) node {111};
			\draw[anchor = north] (6.8,5) node {112};
			\draw[anchor = west] (8.6,6) node {120};
			\draw[anchor = north] (8,5) node {121};
			\draw[anchor = north] (9.2,5) node {122};
			% 2 S_3^2
			\draw (12.2,8)--(11.6,7)--(12.8,7)--cycle;
			\draw (11,6)--(10.4,5)--(11.6,5)--cycle;
			\draw (13.4,6)--(12.8,5)--(14,5)--cycle;
			\draw (11.6,7)--(11,6);
			\draw (11.6,5)--(12.8,5);
			\draw (12.8,7)--(13.4,6);
			\draw (12.2,8) [fill=black] circle (\vr);
			\draw (11.6,7) [fill=black] circle (\vr);
			\draw (12.8,7) [fill=black] circle (\vr);
			\draw (11,6) [fill=black] circle (\vr);
			\draw (10.4,5) [fill=white] circle (\vr);
			\draw (11.6,5) [fill=black] circle (\vr);
			\draw (13.4,6) [fill=black] circle (\vr);
			\draw (12.8,5) [fill=white] circle (\vr);
			\draw (14,5) [fill=white] circle (\vr);
			\draw[anchor = west] (12.2,8) node {200};
			\draw[anchor = east] (11.6,7) node {201};
			\draw[anchor = west] (12.8,7) node {202};
			\draw[anchor = east] (11,6) node {210};
			\draw[anchor = north] (10.4,5) node {211};
			\draw[anchor = north] (11.6,5) node {212};
			\draw[anchor = west] (13.4,6) node {220};
			\draw[anchor = north] (12.8,5) node {221};
			\draw[anchor = north west] (13.5,5) node {222};
			%
			%\draw (7.95,8.25) node {$e_{01}^{(3)}$};
			%\draw (11.60,8.25) node {$e_{02}^{(3)}$};
			%\draw (9.8,5.35) node {$e_{12}^{(3)}$};
			%
			%$\draw (7.0,8.6)--(12.5,8.6)--(12.5,12.5)--(7.0,12.5)--cycle;
			%\draw (12.0,12.2) node {${\bf 0}S^{2}$};
			
			\draw (16,10)--(16,7);
			\draw (23.5,10)--(23.5,7);
			\draw (18.5,10)--(21,7);
			\draw (18.5,7)--(21,10);
			\draw (18.5,12)--(21,12);
			\draw (18.5,5)--(21,5);
			% 0 S_4^1
			\draw (16,12)--(16,10)--(18.5,10)--(18.5,12)--cycle;
			\draw (16,12)--(18.5,10);
			\draw (16,10)--(18.5,12);
			\draw (16,12) [fill=black] circle (\vr);
			\draw (16,10) [fill=black] circle (\vr);
			\draw (18.5,10) [fill=black] circle (\vr);
			\draw (18.5,12) [fill=black] circle (\vr);
			\draw[anchor = south] (16,12) node {00};
			\draw[anchor = east] (16,10) node {01};
			\draw[anchor = west] (18.5,10) node {02};
			\draw[anchor = south] (18.5,12) node {03};
			% 1 S_4^1
			\draw (16,7)--(16,5)--(18.5,5)--(18.5,7)--cycle;
			\draw (16,7)--(18.5,5);
			\draw (16,5)--(18.5,7);
			\draw (16,7) [fill=black] circle (\vr);
			\draw (16,5) [fill=white] circle (\vr);
			\draw (18.5,5) [fill=black] circle (\vr);
			\draw (18.5,7) [fill=black] circle (\vr);
			\draw[anchor = east] (16,7) node {10};
			\draw[anchor = north] (16,5) node {11};
			\draw[anchor = north] (18.5,5) node {12};
			\draw[anchor = west] (18.5,7) node {13};
			% 3 S_4^1
			\draw (21,12)--(21,10)--(23.5,10)--(23.5,12)--cycle;
			\draw (21,12)--(23.5,10);
			\draw (21,10)--(23.5,12);
			\draw (21,12) [fill=black] circle (\vr);
			\draw (21,10) [fill=white] circle (\vr);
			\draw (23.5,10) [fill=white] circle (\vr);
			\draw (23.5,12) [fill=white] circle (\vr);
			\draw[anchor = south] (21,12) node {30};
			\draw[anchor = east] (21,10) node {31};
			\draw[anchor = west] (23.5,10) node {32};
			\draw[anchor = south] (23.5,12) node {33};
			% 2 S_4^1
			\draw (21,7)--(21,5)--(23.5,5)--(23.5,7)--cycle;
			\draw (21,7)--(23.5,5);
			\draw (21,5)--(23.5,7);
			\draw (21,7) [fill=black] circle (\vr);
			\draw (21,5) [fill=white] circle (\vr);
			\draw (23.5,5) [fill=white] circle (\vr);
			\draw (23.5,7) [fill=black] circle (\vr);
			\draw[anchor = east] (21,7) node {20};
			\draw[anchor = north] (21,5) node {21};
			\draw[anchor = north] (23.5,5) node {22};
			\draw[anchor = west] (23.5,7) node {23};
			\end{tikzpicture}
		\end{center}
		\caption{The Sierpi\'nski graph $S_3^{3}$ and $S_4^{2}$ with vertices of $L_3^{3}$ and $L_4^{2}$ bolded.}
		\label{fig:S3domL}
	\end{figure}
		
To show that $L_p^{n}$ is a $\gamma_{gr}$-sequence, we have to check its length. There are $p^{n-1}$ vertices with the last bit $0$. Now we have to compute how many vertices have the form $\langle x_1x_2\dots x_{n-l}ab^{l-1}\rangle$, where $2 \le l \le n$, $b > a$ and $x_1,\ldots,x_{n-l},a,b \in [p]_0$. If $l$ and all $x_1,\ldots,x_{n-l}$ are fixed, then there are $\frac{p(p-1)}{2}$ such vertices. So altogether we have
\begin{equation}
		\sum\limits_{l=2}^{n} \big( p^{n-l} \cdot \frac{p(p-1)}{2}\big) = \frac{p(p-1)}{2} \sum\limits_{l=2}^{n}  p^{n-l}= \frac{p(p-1)}{2} \cdot \frac{p^{n-1}-1}{p-1} = \frac{p(p^{n-1}-1)}{2}
\end{equation}
vertices that satisfy those conditions. This implies that the length of $L_p^{n}$ is $p^{n-1} + \frac{p(p^{n-1}-1)}{2}$, so $L_p^{n}$ is indeed a $\gamma_{gr}$-sequence of $S_p^{n}$. In Fig.~\ref{fig:S3domL},  $S_3^{3}$ and $S_4^{2}$ with $L_3^{3}$ and $L_4^{2}$ are illustrated.

\section{Dominating sequences of interval graphs}
\label{s:interval}
In this section we present an algorithm that generates a Grundy dominating sequence of an arbitrary interval graph. We will use the results of Section~\ref{s:deliting} concerning the deletion of simplicial vertices and twins.

An \emph{interval representation} of a graph is a family of intervals of the real line assigned to vertices so
that vertices are adjacent if and only if the corresponding intervals intersect. A graph is an \emph{interval graph}
if it has an interval representation. For more detailes on interval graphs see \cite{bls-99,mcmc-99}. 

Let us present the ordering and establish the notation of vertices in an interval graph.
Let $G=(V,E)$ be an interval graph with an interval representation $I_G: V(G) \to \{[a,b];~a,b\in {\mathbb{R}},~a \leq b\}$, and vertices $V=\{v_1,\ldots , v_n\}$ sorted in the non-decreasing order according to the right endpoints of corresponding intervals. In other words, $I_G(v_i)=\left[ a_i,b_i \right]$, and $b_1 \leq b_2 \leq \ldots \leq b_n.$ It is clear that $v_1$ is a simplicial vertex of $G$. Let $\widehat{A}=\{a_1,b_1,\ldots a_n,b_n\}$ be the (multi)set of interval endpoints. We will also make use of the non-decreasing sequence $A_{I_G}$ of the real numbers from $\widehat{A}$ of length $2n$, such that all elements of $\widehat{A}$ are used; in the case $a_i = b_j$, for some $i,j \in \{1,2,\ldots,n \}$, $a_i$ is in the sequence before $b_j$. We call the sequence $A_{I_G}$ the \emph{interval endpoints sequence}. In Fig.~\ref{fig:interval1} an example of interval representation and interval endpoints sequence is presented.

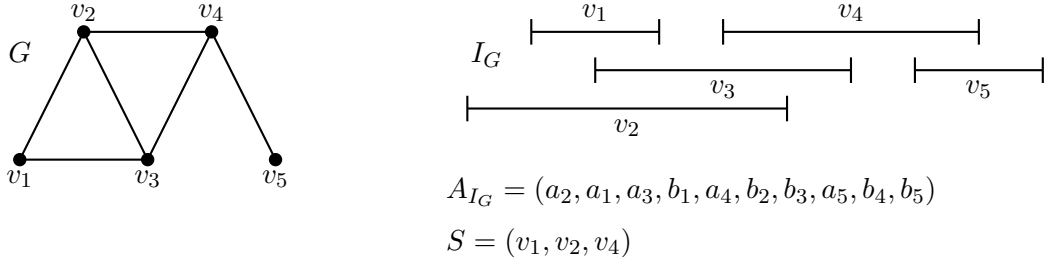
\begin{figure}[ht!]
	
	\begin{center}
		\begin{tikzpicture}[scale=0.85,style=thick,x=1cm,y=1cm]
		\def\vr{2.5pt} % \vr = vertex radius;
		\draw (-8,0)--(-7,2);
		\draw (-8,0)--(-6,0);
		\draw (-7,2)--(-6,0);
		\draw (-7,2)--(-5,2);
		\draw (-6,0)--(-5,2);
		\draw (-4,0)--(-5,2);
		\draw (-7,2) [fill=black] circle (\vr);
		\draw (-8,0) [fill=black] circle (\vr);
		\draw (-6,0) [fill=black] circle (\vr);
		\draw (-5,2) [fill=black] circle (\vr);
		\draw (-4,0) [fill=black] circle (\vr);
		\draw[anchor = north] (-8,0) node {$v_1$};
		\draw[anchor = south] (-7,2) node {$v_2$};
		\draw[anchor = north] (-6,0) node {$v_3$};
		\draw[anchor = south] (-5,2) node {$v_4$};
		\draw[anchor = north] (-4,0) node {$v_5$};
		\draw[anchor = north] (-8,2) node {$G$};

		\draw[anchor = north] (-0.7,2) node {$I_G$};
		\draw (0,2)--(2,2);
		\draw[anchor = south] (1,2) node {$v_1$};
		\draw (0,2.2)--(0,1.8);
		\draw (2,2.2)--(2,1.8);
		
		\draw (-1,0.8)--(4,0.8);
		\draw[anchor = north] (1.5,0.8) node {$v_2$};
		\draw (-1,1)--(-1,.6);
		\draw (4,1)--(4,.6);
		
		\draw (1,1.4)--(5,1.4);
		\draw[anchor = north] (3,1.4) node {$v_3$};
		\draw (1,1.6)--(1,1.2);
		\draw (5,1.6)--(5,1.2);
		
		\draw (3,2)--(7,2);
		\draw[anchor = south] (5,2) node {$v_4$};
		\draw (3,2.2)--(3,1.8);
		\draw (7,2.2)--(7,1.8);
		
		\draw (6,1.4)--(8,1.4);
		\draw[anchor = north] (7,1.4) node {$v_5$};
		\draw (6,1.2)--(6,1.6);
		\draw (8,1.2)--(8,1.6);
		
		\draw[anchor = west] (-1.5,-0.5) node {$A_{I_G} = (a_2,a_1,a_3,b_1,a_4,b_2,b_3,a_5,b_4,b_5)$};
		\draw[anchor = west] (-1.5,-1.3) node {$S = (v_1,v_2,v_4)$};
		\end{tikzpicture}
	\end{center}
	\caption{An interval graph $G$ with interval representation $I_G$, interval endpoints sequence $A_{I_G}$ and Grundy dominating sequence $S$ returned by Algorithm~\ref{al:grundy}.}
	\label{fig:interval1}
\end{figure}

\begin{algorithm}[hbt!]
{\caption{Grundy dominating sequence of an interval graph $G$.}
\label{al:grundy}
\KwIn{An interval graph $G$ with vertices $(v_1,v_2,\ldots , v_n)$ (where $v_i$ corresponds to $[a_i,b_i]$), ordered according to their right end-points, and the interval endpoints sequence $A_{I_G}$.}
\KwOut{Grundy dominating sequence $S$ of a graph $G$.}
\BlankLine
{
	    $S=()$;\\
		$newInterval$ = false;\\
		$A=A_{I_G}$; \\
	\While{$A \neq ()$}
	{
		Choose $e\in A$ such that $A=(e)\oplus A'$;\\

		\If{$e$ is some $a_i$}{
			$newInterval$ = true;\\
		}
		\ElseIf{$e$ is some $b_i$ and $newInterval$ is true}{
			$S=S\oplus (v_i)$;\\
			$newInterval$ = false;\\
		}
		$A=A'$
	}
}
}
\end{algorithm}

\begin{theorem}
Algorithm~\ref{al:grundy} returns a Grundy dominating sequence of an interval graph $G$. In addition, $\ggr(G)$ equals the number of consecutive subsequences of the form $(a_i,b_j)$ in the interval endpoints sequence $A_{I_G}$ for any interval representation $I_G$ of $G$.
\end{theorem}
\begin{proof}
The proof goes by induction on the number of vertices of an interval graph $G)$. For $|V|=1$ the length of the sequence of our algorithm is 1 which is clearly optimal. 

Suppose now that the algorithm returns a Grundy dominating sequence for any interval graph with at most $n-1$ vertices and let $G$ be an arbitrary interval graph with $n$ vertices. Let $V(G)=\{v_1,\ldots , v_n\}$ and let $I_G: V(G) \to \{\left[a,b\right];~a,b \in {\mathbb{R}}, ~a\leq b\}$ be an interval representation of $G$ with $I_G(v_i)=\left[a_i,b_i\right]$ ordered according to their right end-points, i.e., $b_1 \leq b_2 \leq \ldots \leq b_n$. Let $A_{I_G}$ be the corresponding interval endpoints sequence. Since $G-v_1$ is an interval graph with $n-1$ vertices, the algorithm returns the Grundy dominating sequence $S'$ in $G-v_1$ (using the induction hypothesis). Since the vertex with the smallest right end-point is the first vertex in the sequence produced by the algorithm, $v_2$ is the first vertex of $S'.$

There are two options for $v_1$. First, suppose that $v_1$ is a twin vertex in $G$. Then $v_1$ and $v_2$ are twins, and, using Proposition~\ref{prp:simtwin}(ii), $\ggr(G)=\ggr(G-v_1)$. Hence the sequence $S'$ is also a Grundy dominating sequence in $G$. As $v_1$ and $v_2$ are twins, the sequence $S$ obtained from $S'$ by replacing $v_2$ with $v_1$ is also a Grundy dominating sequence in $G$. Note that $S$ is exactly the sequence returned by Algorithm~\ref{al:grundy}. Indeed, since $v_1$ and $v_2$ are two
consecutive vertices with respect to the right endpoints ordering and are twins, no interval endpoint lies between $b_1$ and $b_2$; i.e., $b_1$ and $b_2$ are consecutive endpoints in the sequence $A_{I_G}$. Hence, after $v_1$ is put to $S$ (and $v_2$ is not), the algorithm follows the same steps as the algorithm in $G-v_1$. The proof of this case is complete.

Finally suppose that $v_1$ is not a twin in $G$. Since $v_1$ is simplicial in $G$, by Proposition~\ref{prp:simtwin}(i) we infer $\ggr(G) \leq \ggr(G-v_1)+1.$ As $v_2$ is the first vertex of $S'$ and $N[v_1] \subsetneq  N[v_2]$, $S=(v_1)\oplus S'$ is a legal dominating sequence in $G$. Proposition~\ref{prp:simtwin}(i) again implies that $\ggr(G)=\ggr(G-v_1)+1$, which means that $S$ is a Grundy dominating sequence in $G$. Since $S$ is the sequence returned by Algorithm~\ref{al:grundy}, the proof of the correctness of the algorithm is complete.       

For the second statement in the theorem, one just needs to note that the algorithm counts the number of consecutive subsequences of $A_{I_G}$, in which the first vertex is a left endpoint $a_i$ and the second vertex is a right endpoint $b_j$. Since the Grundy domination number of $G$ is independent from the choice of its interval representation $I_G$ we infer that the number of such subsequnces is also invariant under the interval representation.
\qed
\end{proof}

Let $G$ be an arbitrary interval graph on $n$ vertices and $m$ edges. It is easy to see that the time complexity of Algorithm \ref{al:grundy} is $O(n)$, since the length of the interval endpoints sequence is $2n$. 
It is known that the time complexity of constructing interval representation of an interval graph is $O(n+m)$ \cite{bl-1976,hm-1999,cos-2009}. To sort vertices according to their right endpoints and to construct interval endpoints sequence $O(n\log n)$ time is needed, since we just need to sort endpoints of intervals.
Thus the time complexity of preparing input data for Algorithm~\ref{al:grundy} is $ O(n\log n + m)$. We derive the following result. 

\begin{corollary}
	Let $G$ be an interval graph on $n$ vertices and $m$ edges. The time complexity of preprocessing input data for Algorithm~\ref{al:grundy} is $O(n\log(n) + m)$, and Algorithm~\ref{al:grundy} is linear with complexity $O(n)$.
\end{corollary}

\section*{Acknowledgements}
The authors are supported in part by the Slovenian Research Agency (ARRS); B.B. under the grants P1-0297 and  J1-7110, T.G. under the grants L7-5459 and J1-7110, and
T.K. under the grant P1-0297.

\end{document}